\documentclass[leqno,twoside, 11pt]{amsart}
\usepackage{amssymb, latexsym, esint}
\usepackage{color}

\theoremstyle{plain}
\numberwithin{equation}{section} \numberwithin{figure}{section}

\newtheorem{theorem}{Theorem}[section]
\newtheorem{lemma}[theorem]{Lemma}

\newtheorem{corollary}[theorem]{Corollary}
\newtheorem{definition}[theorem]{Definition}
  \usepackage{epsfig} 
\usepackage{graphicx}  \usepackage{epstopdf}
\theoremstyle{definition}
\newtheorem{remark}[theorem]{Remark}

\begin{document}
\title[Non-homogeneous p-Laplace equations] {On viscosity and weak solutions for non-homogeneous p-Laplace equations}
\author{Mar\'ia Medina and Pablo Ochoa}
\address{Maria Medina, Facultad de Matem\'aticas, Pontificia Universidad Cat\'olica de Chile, Avenida Vicu\~na Mackenna 4860, Santiago, Chile}
\email{mamedinad@mat.puc.cl}
\address{Pablo Ochoa, Universidad Nacional de Cuyo-CONICET, Mendoza 5500, Argentina} 
\email{ochopablo@gmail.com}



\thanks{The first  author was supported by the grant FONDECYT Postdoctorado 2016, No. 3160077. The second author was partially supported by CONICET and grant PICT 2015-1701 AGENCIA.}

\date{\today}
\begin{abstract} 

In this manuscript, we study the relation between viscosity and weak solutions for non-homogeneous p-Laplace equations with lower order term depending on $x$, $u$ and $\nabla u$. More precisely, we prove that any locally bounded viscosity solution constitutes a weak solution, extending  results presented in Juutinen, Lindqvist and Manfredi \cite{JLM}, and Julin and Juutinen \cite{JJ}. Moreover, we provide a converse  statement in the full case under extra assumptions  on the data.

\end{abstract}

\maketitle

\section{Introduction and main results}

In this work, we consider the following degenerate (or singular) elliptic equations of p-Laplacian type
\begin{equation}\label{1}
-\mbox{div}\big( |\nabla u|^{p-2}\nabla u\big)=f(x, u, \nabla u),
\end{equation}defined in an open and bounded set $\Omega \subset \mathbb{R}^{n}$ and for $1 < p< \infty$. The modulus of ellipticity of the p-Laplace operator is $|\nabla u|^{p-2}$. When $p > 2$, the modulus vanishes whenever $\nabla u=0$ and the equation is called degenerate at those points where that occurs. On the other hand, for $p< 2$, the modulus becomes infinite when $\nabla u=0$, and the equation is called singular at those points. Observe that the case $p=2$ is just the linear case and corresponds to the Laplace operator.

Different notions of solutions have been formulated for equation \eqref{1}. We are interested in the relation between Sobolev weak solutions and viscosity solutions. For the homogeneous p-Laplace equation, this relation has already been studied by Juutinen, Lindqvist and Manfredi in \cite{JLM}, via the notion of p-harmonic, p-subharmonic and p-superharmonic functions. Roughly speaking, a p-harmonic function is a continuous function which solves, weakly, the homogeneous p-Laplace equation, and a p-superharmonic (p-subharmonic) function is a lower (upper) semicontinuous function that admits comparison with p-harmonic functions from below (above). 

In \cite{JLM}, the authors showed that the notion of p-harmonic solution is equivalent to the notion of viscosity solution. Moreover, it was shown in \cite{L} that locally bounded p-harmonic functions are weak solutions. Conversely, every weak solution to the homogeneous p-Laplace equation has a representative which is lower semi-continuous and it is p-harmonic. We refer the interested reader to \cite{H} for further details. In this way, there is an equivalence between the notion of weak and viscosity solutions for the homogeneous framework. It is worth to mention that a different and simpler proof of this equivalence was stated by Julin and Juutinen in \cite{JJ} by using inf and sup convolutions.  In turn, this reasoning was extended in \cite{N} to more general second-order differential equations.

For the non-homogeneous case, the  notion of p-harmonic functions is lost  and we need to study directly the link between viscosity and Sobolev weak solutions. In \cite{JJ}, the authors showed that viscosity solutions of \eqref{1} are weak solutions in the case where $f$ is continuous and  depends only on $x$. 

Our main goal in the present manuscript is to prove the equivalence of these two notions of solutions for the general structure \eqref{1}. The implication that viscosity solutions are weak solutions is partially based on the work \cite{JJ}, but the non-homogeneous nature of the equation under consideration requires some extra effort to deal with the lower order term. 

On the other hand, the converse statement relies on comparison principles for weak solutions. To the best of our knowledge, the available comparison results for the full case $f=f(x, s, \eta)$ require additional limitations in the degenerate case which do not appear in the singular context (compare Theorem \ref{comparison for weaks} and Theorem \ref{comparison for weaks singular}). Moreover, we believe that  the assumption that weak subsolutions and weak supersolutions belong to $\mathcal{C}^{1}$ or to the Sobolev space $W^{1, \infty}_{loc}$ in order to have comparison is not a strong limitation since we are interested in the equivalence of weak and viscosity solutions, and for weak solutions the $\mathcal{C}^{1, \alpha}$-regularity holds (see \cite{DB, P}). Finally, in the quasi-linear case $f=f(x, u)$ there is no need to impose higher regularity than $W^{1, p}_{loc} \cap \mathcal{C}$ on the solutions. We refer the reader to \cite{PS} for a survey of maximum principles and comparison results for general structures in divergence form.

Finally, we stress that the equivalence between weak and viscosity solutions may be used to prove relevant properties on the solutions. As an example, in \cite{JLr} the authors prove a Rad\'o's type theorem for p-harmonic functions. Roughly speaking, they state that if a function $u$ solves, weakly, the homogeneous p-Laplace equation in the complement of the set where $u$ vanishes, then it is a solution in the whole set. It is an open problem to obtain a similar result for equations like \eqref{1}. We shall return to this issue in a subsequent paper.
\medskip

We recall that the p-Laplace operator is defined as
$$\Delta_p u:=\mbox{div}\big( |\nabla u|^{p-2}\nabla u\big).$$
Let us state the different type of solutions to \eqref{1} we will manage.

\begin{definition}[Sobolev weak solution] A function $u \in W^{1, p}_{loc}(\Omega)$ is a weak supersolution to \eqref{1} if
\begin{equation}
\int_\Omega|\nabla u|^{p-2}\nabla u \cdot \nabla \psi \geq \int_\Omega f(x, u, \nabla u)\psi
\end{equation}
for all non negative $\psi \in \mathcal{C}_0^{\infty}(\Omega)$. On the other hand, $u$ is a weak subsolution if  $-u$ is a weak supersolution of the equation $-\Delta_p u=-f(x, -u, -\nabla u)$. We call $u$ a weak solution if it is both a weak subsolution and a weak supersolution to \eqref{1}.
\end{definition}

Due to the non-homogeneous nature of \eqref{1}, viscosity solutions are stated as in \cite{JJ}, considering semicontinuous envelopes of the p-Laplace operator. More precisely, we have

\begin{definition}
A lower semicontinuous function $u:\Omega\rightarrow (-\infty,+\infty]$ is a viscosity supersolution to \eqref{1} if $u\not\equiv +\infty$ and for every $\phi\in\mathcal{C}^2(\Omega)$ such that $\phi(x_0)=u(x_0)$, $u(x)\geq \phi(x)$ and $\nabla \phi(x)\neq 0$ for all $x\neq x_0$, there holds
\begin{equation}\label{limsup}
\lim_{r\rightarrow 0}\sup_{x\in B_r(x_0)\setminus\{x_0\}}(-\Delta_p\phi(x))\geq f(x_0,u(x_0),\nabla \phi(x_0)).
\end{equation}
A function $u$ is a viscosity subsolution if $-u$ is a viscosity supersolution to the equation $-\Delta_p u=-f(x, -u, -\nabla u)$ , and it is a viscosity solution if it is both a viscosity sub- and supersolution.
\end{definition}
\begin{remark}
Notice that condition \eqref{limsup} is established this way to avoid the problems derived from having $\nabla\phi(x_0)=0$ in the case $1<p<2$. If $p\geq 2$, this condition can be simply replaced by
$$-\Delta_p\phi(x_0)\geq f(x_0,u(x_0),\nabla\phi(x_0)).$$
\end{remark}

We now list the main contributions of our work. The results are stated for supersolutions, but they hold for subsolutions as well. 


\begin{theorem}\label{nonSing}
Let $1 < p < \infty$. Assume that  $f=f(x,s,\eta)$ is uniformly continuous in $\Omega \times \mathbb{R}\times \mathbb{R}^{n}$, non-increasing in $s$, and satisfies the following growth condition
\begin{eqnarray}\label{growth cond}
|f(x,s, \eta)|\leq \gamma(|s|)|\eta|^{p-1} + \phi(x),\end{eqnarray}  
where $\gamma \geq 0$ is  continuous, and $\phi \in L^{\infty}_{loc}(\Omega)$. Hence, 
if $u \in L^{\infty}_{loc}(\Omega)$ is a viscosity supersolution to \eqref{1}, then it is a weak supersolution to \eqref{1}.
\end{theorem}

A  converse of Theorem \ref{nonSing} is given below.

\begin{theorem}\label{converse}
Assume that  $f=f(x,s,\eta)$  is continuous in $\Omega \times \mathbb{R}\times \mathbb{R}^{n}$, non-increasing in $s$, and  locally Lipschitz continuous with respect to  $\eta$. Hence we have the following:

\begin{itemize}
\item[(i)] If $1 <p \leq 2$ and  if $u \in  W^{1, \infty}_{loc}(\Omega)$ is a weak supersolution to \eqref{1}, then it is a viscosity supersolution to \eqref{1} in $\Omega$.
\item[(ii)] If $p >2$, $f(x, s, 0)=0$ for $x \in \Omega$ and $s \in \mathbb{R}$, and if $u \in \mathcal{C}^{1}(\Omega)$ is a weak supersolution to \eqref{1}, then it is a viscosity supersolution to \eqref{1} in $\Omega$.
\item[(iii)] Finally, if $p >2$ and if $u \in  W^{1, \infty}_{loc}(\Omega)$ is a weak supersolution to \eqref{1}, with $\nabla u \neq 0$ in $\Omega$, then it is a viscosity supersolution to \eqref{1}.
\end{itemize}
\end{theorem}

\begin{remark}According to recent results (see \cite{Tom}), it is possible to weak the locally Lipschitz assumption in Theorem \ref{converse} when $f$ takes some particular forms or it satisfies extra convexity and coercivity assumptions. For instance, as a consequence of the results in \cite{Tom}, if
$$f(x, \eta)= c|\eta|^{q}+ \phi(x), \quad \phi > 0,\,c\in \mathbb{R},\,\eta \in \mathbb{R}^{n} \textnormal{ and }q \in [1,p],$$then Theorem \ref{converse} holds for bounded supersolutions $u$ in $W^{1, p}(\Omega)\cap \mathcal{C}(\Omega)$. It is also a consequence of \cite[Th. 1.3]{Tom}, that the same conclusion is obtained when 
$$f(x, s, \eta)=h(s)|\eta|^{p-1}+\phi(x)$$where $h \geq 0$ is decreasing and $\phi \geq 0$.
\end{remark}

In view of the available regularity theory for weak solution of \eqref{1}, we have the following equivalence.
\begin{corollary}Let $1 < p < \infty$. Assume that  $f=f(x,s,\eta)$ is uniformly continuous, locally Lipschitz in $\eta$, non-increasing in $s$ and satisfies the growth condition \eqref{growth cond}. Additionally, assume that $f(x, s, 0)=0$ for $x\in \Omega$ and $s \in \mathbb{R}$ when $p >2$. Then  $u$ is a weak solution to \eqref{1} if and only if it is a viscosity solution to \eqref{1}. 
\end{corollary}

We point out that, in the degenerate case, it is possible to remove the assumption $f(x,s,0)=0$ by imposing the non-vanishing of the gradient of the weak solution in the whole $\Omega$. This is a straightforward consequence of Theorem \ref{converse} (iii). 

In the particular case where $f$ does not depend on $\eta$, we have the following converse to Theorem \ref{nonSing}   which does not require the locally-Lipschitz regularity of the solutions.

\begin{theorem}\label{nice converse}
Let $1 < p < \infty$. Suppose that $f=f(x, s)$ is continuous in $\Omega \times \mathbb{R}$ and non-increasing in $s$.   If $u \in  W^{1, p}_{loc}(\Omega) \cap \mathcal{C}(\Omega)$ is a weak supersolution to \eqref{1}, then it is a viscosity supersolution to \eqref{1}.
\end{theorem}
\medskip

Let us briefly discuss the above hypotheses on $f$. Firstly,  assuming that $f$ is non-increasing and introducing the operator
\begin{equation}
F(x, s, \eta, \mathcal{X}) :=-|\eta|^{p-2}\left( \mbox{tr}(\mathcal{X})+\frac{p-2}{|\eta|^{2}}\mathcal{X}\eta \cdot \eta \right)-f(x, s, \eta), \quad p\geq 2,
\end{equation}we derive that  $F$ is proper, that is, $F$ is non-increasing in $\mathcal{X}$ and non-decreasing in $s$, which is a standard and useful assumption in the theory of viscosity solutions \cite{CIL}. For instance, it allows to get the equivalence between classical solutions ($\mathcal{C}^{2}$ functions which satisfy the equations pointwise) and  $\mathcal{C}^{2}$ viscosity solutions. On the other hand, the growth property \eqref{growth cond} implies the $\mathcal{C}^{1, \alpha}$-regularity of weak solutions to \eqref{1} (see \cite{DB, RZ, P}). Moreover, under a regular Dirichlet boundary condition $\varphi \in \mathcal{C}^{1, \alpha}$, it follows the $\mathcal{C}^{1, \alpha}$-regularity up to the boundary of weak solutions. For further details, see the reference \cite{Lib}. Finally, the extra assumption $f(x, s, 0)=0$ appearing in Theorem \ref{converse} in the degenerate case is used to remove critical sets of points of the weak solution (see reference \cite{JL2}). Hence, it allows the application of comparison results without assuming the non-vanishing of the gradients. We point out that other properties of $f=f(x,s,\eta)$, as more regularity on $s$ and $\eta$ and convexity-like conditions, may be employed to ensure comparison for weak solutions. We refer the reader to \cite{Tom} and the references therein for more details. 

It is worth to mention that many equations appearing in the literature have the structure of \eqref{1} with lower order term satisfying the above assumptions on $f$.  We refer the reader to \cite{MMS, PS, DS, CT} and the references therein for examples of such $f$. 

\medskip

The paper is organized as follows: in Section \ref{prel} we provide some preliminary results concerning properties of infimal convolutions (which will be the main tool in the proof of Theorem \ref{nonSing}) and a convergence result. In addition, we prove a Caccioppoli type estimate that will provide important uniform bounds, fundamental when using approximation arguments. This result is interesting itself.

Section \ref{theo} contains the proof of the main result of the paper, Theorem \ref{nonSing}, that states under which conditions on the non-homogenous function $f$ in \eqref{1} viscosity solutions are actually weak solutions. This proof is divided into two major cases: the singular and the degenerate scenario, since, although both cases rely on the same idea, different approximations and estimates are needed depending on the range of $p$.

In Section \ref{Proof converse} we prove the reverse statement, that is, weak solutions of \eqref{1} are viscosity solutions. This result is based on comparison arguments, and this will determine the conditions we will need to impose on $f$. Finally, in the Appendix we give, for the sake of completeness,  precise references and state the comparison results that we use in the previous section.

\section{Preliminary results}\label{prel}

\subsection{Infimal convolution} Let us define the infimal convolution of a function $u$ as
\begin{equation}\label{infconv}
u_\varepsilon(x):=\inf_{y\in\Omega}\left(u(y)+\frac{|x-y|^q}{q\varepsilon^{q-1}}\right),
\end{equation}
where $q\geq 2$ and $\varepsilon >0$. 

We recall some useful properties of $u_\varepsilon$. Let $u: \Omega \to \mathbb{R}$ be bounded and lower semicontinuous in $\Omega$. It is well-known that  $u_\varepsilon$ is an increasing sequence of  semiconcave  functions in $\Omega$,  which converges pointwise to $u$. Hence, $u_\varepsilon$ is locally Lipschitz and  twice differentiable a.e. in $\Omega$. Moreover, it is possible to write

$$u_\varepsilon(x)= \inf_{y \in B_{r(\varepsilon)}(x)\cap \Omega}\Big( u(y)+ \dfrac{|x-y|^{q}}{q\varepsilon^{q-1}}\Big),$$for $r(\varepsilon) \to 0$ as $\varepsilon \to 0$. For these and further properties see  \cite[Lemma A.1.]{JJ} and \cite{CIL}.

The next lemma is the counterpart of  \cite[Lemma A.1. (iii)]{JJ}  for our setting.

\begin{lemma} \label{infConvLemma}
Suppose that $u: \Omega \to \mathbb{R}$ is bounded and lower semicontinuous in $\Omega$. Let $f=f(x, s, \eta)$ be continuous in $\Omega \times \mathbb{R}\times \mathbb{R}^{n}$ and non-increasing in $s$. If $u$ is a viscosity supersolution to
\begin{equation*}\label{eq 12}
-\Delta_pu=f(x, u, \nabla u)
\end{equation*}
in $\Omega$ for $1<p<\infty$, then $u_\varepsilon$ is a viscosity supersolution to
$$-\Delta_pu_\varepsilon=f_\varepsilon(x, u_\varepsilon, \nabla u_\varepsilon)$$
in $\Omega_\varepsilon:=\{x \in \Omega: \textnormal{dist}(x, \partial \Omega)>r(\varepsilon)\}$, where
$$f_\varepsilon(x, s, \eta):= \inf_{y\in B_{r(\varepsilon)}(x)}f(y, s, \eta).$$

\end{lemma}

\begin{proof}
We start by noticing  that 
$$u_\varepsilon(x)= \inf_{z \in B_{r(\varepsilon)}(0)}\Big( u(z+x)+\frac{|z|^{q}}{q\varepsilon^{q-1}}\Big),\qquad x\in\Omega_\varepsilon.$$
Let us see first that for every $z \in B_{r(\varepsilon)}(0)$, the function
$$\phi_z(x):= u(z+x)+\frac{|z|^{q}}{q\varepsilon^{q-1}}$$
is a viscosity supersolution to $-\Delta_p \phi_z =f_\varepsilon$ in $\Omega_\varepsilon$. Indeed, let $x_0 \in \Omega_\varepsilon$ and $\varphi\in \mathcal{C}^{2}(\Omega_\varepsilon)$ so that
$$\min_{\Omega_\varepsilon}(\phi_z-\varphi ) = ( \phi_z-\varphi)(x_0)=0.$$
We assume that $\nabla \varphi(x) \neq 0$ for all $x\neq x_0$ if $1<p<2$. Making $y:=z+x$, $y_0:=z+x_0$ and $$\tilde{\varphi}(y):= \varphi(y-z)-\frac{|z|^{q}}{q\varepsilon^{q-1}},$$ we derive that $u-\tilde{\varphi}$ has a local minimum at $y_0$, and indeed $(u-\tilde{\varphi})(y_0)=0$. Since $u$ is a viscosity supersolution to \eqref{eq 12}, it follows
$$\lim_{\rho\rightarrow 0}\sup_{x\in B_\rho(y_0)\setminus\{y_0\}}(-\Delta_p\tilde{\varphi}(x))\geq f(y_0, \tilde{\varphi}(y_0), \nabla \tilde{\varphi}(y_0)).$$Therefore,
\begin{equation}\label{inequality}
\begin{split}
\lim_{\rho\rightarrow 0}\sup_{x\in B_\rho(x_0)\setminus\{x_0\}}(-\Delta_p\varphi(x))&= \lim_{\rho\rightarrow 0}\sup_{x\in B_\rho(y_0)\setminus\{y_0\}}(-\Delta_p\tilde{\varphi}(x))\\&\geq f(y_0, \tilde{\varphi}(y_0), \nabla \tilde{\varphi}(y_0))\\ &= f(z+x_0, \tilde{\varphi}(z+x_0), \nabla \varphi(x_0))\\ & = f\left(z+x_0, \varphi(x_0)-\frac{|z|^{q}}{q\varepsilon^{q-1}}, \nabla \varphi(x_0)\right) \\& \geq  f\left(z+x_0, \varphi(x_0), \nabla \varphi(x_0)\right)\\ &\geq f_\varepsilon(x_0, \varphi(x_0), \nabla \varphi(x_0)),
\end{split}
\end{equation}where we have used that $f$ is non-increasing in the second variable. Let us see now that, since $u_\varepsilon$ is an infimum of supersolutions, it is itself a supersolution (observe that $u_\varepsilon$ is continuous, since it is locally Lipschitz).  Let $x_0 \in \Omega_\varepsilon$ and $\phi \in \mathcal{C}^{2}(\Omega_\varepsilon)$ so that
\begin{equation}\label{min}
\min_{\Omega_\varepsilon}(u_\varepsilon-\phi)=(u_\varepsilon-\phi)(x_0)=0.
\end{equation}
Again, $\nabla \phi(x)\neq 0$ for all $x\neq x_0$ in the singular scenario. Moreover, we may assume that the minimum is strict.  For each $n$, there exists $z_n \in B_{r(\varepsilon)}(0)$ such that

\begin{equation}\label{unon}
u(z_n+x_0)+ \frac{|z_n|^{q}}{q\varepsilon^{q-1}}<u_\varepsilon(x_0)+ \frac{1}{n}.
\end{equation} 
Let $x_n$ be a sequence of points in $\overline{B}_r(x_0) \subset \Omega_\varepsilon$ so that
\begin{equation*}\label{minpn}
u(z_n+x_n)+ \frac{|z_n|^{q}}{q\varepsilon^{q-1}}-\phi(x_n)\leq u(z_n+x)+ \frac{|z_n|^{q}}{q\varepsilon^{q-1}}-\phi(x)
\end{equation*}
for all $x \in \overline{B}_r(x_0)$, i.e., $(\phi_{z_n}-\phi)$ has a minimum in $\overline{B}_r(x_0)$ at $x_n$. Up to a subsequence, $x_n \to y_0$ as $n\to \infty$. Furthermore, by \eqref{unon},
\begin{equation}\begin{split}\label{u phi}
u_\varepsilon(x_n)-\phi(x_n)&\leq u(z_n+x_n)+ \frac{|z_n|^{q}}{q\varepsilon^{q-1}} -\phi(x_n)\\
&\leq u(z_n+x_0) + \frac{|z_n|^{q}}{q\varepsilon^{q-1}}-\phi(x_0)\\
&\leq u_\varepsilon(x_0)+\frac{1}{n}-\phi(x_0).
\end{split}\end{equation}
Taking liminf and using the lower semicontinuity of $u_\varepsilon$, we derive
$$u_\varepsilon(y_0)-\phi(y_0)\leq u_\varepsilon(x_0)-\phi(x_0).$$
Since the minimum in \eqref{min} is strict, we must have $y_0=x_0$. Moreover, taking 
$$\varphi(x):=\phi(x)+(\phi_{z_n}-\phi)(x_n),$$
in  \eqref{inequality} we have
$$\lim_{\rho\rightarrow 0}\sup_{x\in B_\rho(x_n)\setminus\{x_n\}}(-\Delta_p\phi(x))\geq f\Big(z_n+x_n, u(z_n+x_n)+\frac{|z_n|^{q}}{q\varepsilon^{q-1}}, \nabla \phi(x_n)\Big) .$$
Since $f$ is non increasing with respect to the second variable, by \eqref{u phi} we obtain
$$\lim_{\rho\rightarrow 0}\sup_{x\in B_\rho(x_n)\setminus\{x_n\}}(-\Delta_p\phi(x))\geq f\Big(z_n+x_n, u_\varepsilon(x_0)+\frac{1}{n}-\phi(x_0)+\phi(x_n), \nabla \phi(x_n)\Big).$$As $n\to \infty$, there holds
$$\lim_{\rho\rightarrow 0}\sup_{x\in B_\rho(x_0)\setminus\{x_0\}}(-\Delta_p\phi(x))\geq f(z'+x_0, u_\varepsilon(x_0), \nabla \phi(x_0)),$$for some $z' \in \overline{B_r(0)}$. Therefore
$$\lim_{\rho\rightarrow 0}\sup_{x\in B_\rho(x_0)\setminus\{x_0\}}(-\Delta_p\phi(x))\geq f_\varepsilon(x_0, \phi(x_0), \nabla \phi(x_0)),$$ 
and we conclude that $u_\varepsilon$ is a viscosity supersolution of
$$-\Delta_pu_\varepsilon=f_\varepsilon(x, u_\varepsilon, \nabla u_\varepsilon)\,\hbox{ in }\Omega_\varepsilon.$$
\end{proof}

The next lemma states the weak convergence of the lower order terms in the particular situation of infimal convolutions.

\begin{lemma}\label{convf}
Let $f=f(x,s,\eta)$ be a uniformly continuous function, which satisfies the growth condition \eqref{growth cond}.  Assume that $u \in W^{1, p}_{loc}(\Omega)$ is locally bounded and lower semicontinuous in $\Omega$. For each $\varepsilon>0$ define $u_\varepsilon$ as in \eqref{infconv} and $f_\varepsilon$ as in Lemma \ref{infConvLemma}. 

Then, if $\nabla u_\varepsilon$ converges to $\nabla u$ in $L^{p}_{loc}(\Omega)$, the following holds
$$\lim_{\varepsilon\rightarrow 0}\int_\Omega{f_\varepsilon(x,u_\varepsilon,\nabla u_\varepsilon)\psi\,dx}=\int_\Omega{f(x,u,\nabla u)\psi\,dx} $$for every non negative $\psi\in\mathcal{C}_0^\infty(\Omega)$.
\end{lemma}

\begin{proof}
Let $\psi\in \mathcal{C}_0^\infty(\Omega)$ and denote $K:=$spt$(\psi)$. Consider $\varepsilon>0$ small enough so that 
$$K\subset K'\subset \Omega,$$
where $K':=\overline{\cup_{x\in K}B_{r(\varepsilon)}(x)}$. Since $f$ is uniformly continuous in $K' \times \mathbb{R} \times \mathbb{R}^{n}$, for every $\rho>0$ there exists $\delta>0$ such that
$$|f(x,u_\varepsilon(x),\nabla u_\varepsilon(x))-f(y,u_\varepsilon(x),\nabla u_\varepsilon(x))|<\rho\qquad \hbox{ if }|x-y|<\delta,\;\;x,y\in K'.$$
Choose $\varepsilon_0>0$ so that $r(\varepsilon)<\delta$ for every $\varepsilon<\varepsilon_0$. Thus, from the previous inequality we get
$$f(x,u_\varepsilon(x),\nabla u_\varepsilon(x))<\rho+f(y,u_\varepsilon(x),\nabla u_\varepsilon(x)),$$
for every $x\in K$ and $y\in B_{r(\varepsilon)}(x)$. In particular,
$$f(x,u_\varepsilon(x),\nabla u_\varepsilon(x))<\rho+f_\varepsilon(x,u_\varepsilon(x),\nabla u_\varepsilon(x)),$$
and therefore
$$0 \leq |f(x,u_\varepsilon(x),\nabla u_\varepsilon(x))-f_\varepsilon(x,u_\varepsilon(x),\nabla u_\varepsilon(x))|<\rho.$$
Hence we arrive at the estimate
\begin{equation}\label{estimate f}
\int_\Omega |f(x,u_\varepsilon,\nabla u_\varepsilon)-f_\varepsilon(x,u_\varepsilon,\nabla u_\varepsilon)|\psi\,dx\leq \rho \|\psi\|_{L^{\infty}(K)}|K|.
\end{equation}On the other hand, due to the continuity of $f$ and the convergences of $u_\varepsilon$ and $\nabla u_\varepsilon$,
$$f(x, u_\varepsilon(x), \nabla u_\varepsilon (x)) \to f(x, u(x), \nabla u(x)) \quad \mbox{a. e. in }\Omega.$$
Observe that
$$u_{\varepsilon_0} \leq u_\varepsilon \leq u,\quad \textnormal{ for all }\varepsilon \leq \varepsilon_0.$$Since $u_{\varepsilon_0}$, $u$ belong to $L^{\infty}_{loc}(\Omega)$, there exists a uniform constant $C>0$ so that
$$\|u_\varepsilon\|_{L^{\infty}(K)}\leq C, \quad \varepsilon \leq \varepsilon_0.$$
Thus, in view of the growth estimate on $f$ and the continuity of $\gamma$, we have, for an appropriate positive constant $C$,
\begin{equation}\label{gf}
|f(x, u_\varepsilon(x), \nabla u_\varepsilon(x))| \leq C |\nabla u_\varepsilon(x)|^{p-1}+\phi(x),
\end{equation}Since $|\nabla u_\varepsilon|^{p-1} \in L^{p/(p-1)}_{loc}(\Omega)$, H\"{o}lder inequality and the strong convergence of $\nabla u_\varepsilon$ imply
$$\int_K |\nabla u_\varepsilon|^{p-1} \leq C \|\nabla u_\varepsilon\|^{p-1}_{L^{p}(K)} \leq C, \mbox{ for all $\varepsilon$.}$$
By \eqref{estimate f}, \eqref{gf} and Lebesgue dominated convergence Theorem, we conclude
\begin{equation*}\label{estimate 2 f}
\lim_{\varepsilon \to 0}\int_K f(x, u_\varepsilon, \nabla u_\varepsilon)\psi\, dx= \int_Kf(x, u, \nabla u)\psi\, dx.
\end{equation*}
\end{proof}

\subsection{A Caccioppoli's estimate}In the next lemma we provide a Caccioppoli's estimate for the $L_{loc}^{p}$-norm of the gradients of weak solutions.

\begin{lemma}\label{caccioppoli} Let $u \in W^{1, p}(\Omega)$ be a locally bounded weak supersolution to \eqref{1}. Assume that  $f$ is continuous in $\Omega \times \mathbb{R}\times \mathbb{R}^{n}$ and satisfies the growth bound \eqref{growth cond}. Then there exists a constant $C=C(p, \Omega, \phi, \gamma) >0$ such that for all test function $\xi \in \mathcal{C}_0^{\infty}(\Omega)$, $0 \leq \xi \leq 1$, we have
\begin{equation*}
\int_\Omega |\nabla u|^{p}\xi^{p}dx \leq C\left[(\textnormal{osc}_{K}u)^{p}\int_\Omega \left(|\nabla \xi|^{p} + 1\right) dx +\textnormal{osc}_{K}u\right],
\end{equation*}where $\textnormal{osc}_{K}u:=\sup_{K}u-\inf_{K}u$, and $K:=$ \textnormal{spt}$(\xi)$.
\end{lemma}
\begin{proof}
Let $\xi \in \mathcal{C}^{\infty}_0(\Omega)$ and $K \subset \Omega$ as in the lemma. Consider the test function
$$\psi(x):=\left(\sup_{K} \,u - u(x)\right)\xi^{p}(x), \qquad x \in \Omega.$$Then,
\begin{equation*}
\begin{split}
&\int_\Omega f(x, u, \nabla u)\psi\,dx \leq  \int_\Omega |\nabla u|^{p-2}\nabla u\cdot \nabla \psi\,dx\\& \qquad \qquad \qquad \qquad=-\int_\Omega |\nabla u|^{p-2}\nabla u\cdot \left[ \xi^{p}\nabla u- p  \xi^{p-1}\nabla \xi\left(\sup_{K}u - u\right)\right] dx.
\end{split}
\end{equation*}Therefore,
\begin{equation}\label{first int}
\begin{split}
\int_\Omega |\nabla u|^{p}\xi^{p}dx &\leq p  \int_\Omega \xi^{p-1}|\nabla u|^{p-2}\nabla u \cdot\nabla \xi\left(\sup_{K}u - u\right) dx -\int_\Omega f(x, u, \nabla u)\psi\,dx.
\end{split}
\end{equation}Observe that
\begin{equation*}
\int_\Omega ||\nabla u|^{p-2}\nabla u| dx\leq \int_\Omega |\nabla u|^{p-1}dx,
\end{equation*}which shows that $|\nabla u|^{p-2}\nabla u \in L^{p/(p-1)}(\Omega)$. Hence, Young's inequality
$$ab \leq \delta a^{q}+ \delta^{-1/(q-1)}b^{q'},$$where $q$ and  $q'$ are conjugate exponents, implies that the first integral in  the right-hand side of \eqref{first int} may be bounded by
\begin{equation*}
\delta\int_\Omega |\nabla u|^{p}\xi^{p} +\delta^{1-p}\int_\Omega p^{p}|\nabla \xi|^{p}(\textnormal{osc}_{K}u)^{p}dx.
\end{equation*}Moreover, by \eqref{growth cond} we have
\begin{equation}\label{est f epsilon}
f(x, u(x), \nabla u(x)) \geq -\gamma_\infty|\nabla u(x)|^{p-1} - \|\phi\|_{L^{\infty}(K)},
\end{equation}for all $x$ in the support of $\xi$, where $\gamma_\infty:=\sup_{x\in K}{|\gamma(u(x))|}$. Therefore, the second integral in \eqref{first int} is estimated from above by
\begin{equation*}
\gamma_\infty\int_\Omega |\nabla u|^{p-1}\left(\sup_{K}u - u\right)\xi^{p}dx +C(\Omega, \phi)\textnormal{osc}_{K}u,
\end{equation*}where $C(\Omega, \phi)$ is a positive constant. The assumption $\xi \leq 1$ and Young's inequality  yield

\begin{equation*}
\begin{split}
\gamma_\infty\int_\Omega |\nabla u|^{p-1}\left(\sup_{K}u - u\right)\xi^{p}dx &\leq \gamma_\infty\int_\Omega |\nabla u|^{p-1}\left(\sup_{K}u - u\right)\xi^{p-1}dx\\&\leq \delta\int_\Omega|\nabla u|^{p} \xi^{p}dx + \delta^{1-p} C(p, \Omega, \gamma)(\textnormal{osc}_{K}u)^{p}.
\end{split}
\end{equation*}

Therefore,
\begin{equation*}
\begin{split}
&\int_\Omega |\nabla u|^{p}\xi^{p}dx \leq 2\delta\int_\Omega |\nabla u|^{p}\xi^{p}dx + 
\delta^{1-p}C(p, \Omega, \gamma, \phi)\left[(\textnormal{osc}_Ku)^{p}\int_\Omega\left(|\nabla \xi|^{p}+1\right)dx + \textnormal{osc}_Ku\right]
\end{split}
\end{equation*}Taking $\delta < 1/2$, we derive the Caccioppoli's estimate.
\end{proof}

\section{Proof of Theorem \ref{nonSing}}\label{theo}

\subsection{Degenerate case}  We begin with the range $p\geq2$.
\begin{proof}[Proof of Theorem \ref{nonSing}]
 Let $u_\varepsilon$ be the infimal convolution defined in \eqref{infconv} with $q=2$. Then,
$$\phi(x):=u_\varepsilon(x)-C|x|^{2}$$is concave in $\Omega_{r(\varepsilon)}$ (see \cite[Lemma A.2.]{JJ}). By Aleksandrov's Theorem, $\phi$ is twice differentiable almost everywhere in $\Omega_{r(\varepsilon)}$, and so $u_\varepsilon$. Therefore by Lemma \ref{infConvLemma},
$$-\Delta_p u_\varepsilon(x)\geq f_\varepsilon (x, u_\varepsilon(x), \nabla u_\varepsilon (x))$$
a.e. in $\Omega_{r(\varepsilon)}$. Furthermore,
$$\int_\Omega |\nabla u_\varepsilon|^{p-2}\nabla u_\varepsilon \cdot \nabla \psi \geq \int_\Omega (-\Delta_p u_\varepsilon)\psi,$$for all non-negative test function $\psi$  (see the proof of \cite[Theorem 3.1]{JJ}). Hence, we derive
\begin{equation*}
\int_{\Omega}f_\varepsilon(x, u_\varepsilon, \nabla u_\varepsilon)\psi\,dx \leq \int_\Omega |\nabla u_\varepsilon|^{p-2}\nabla u_\varepsilon\cdot \nabla \psi\, dx,
\end{equation*}for all non-negative test function $\psi$ and all $\varepsilon >0$. We claim that, as $\varepsilon \to 0$, there holds

\begin{equation*}
\int_{\Omega}f(x, u, \nabla u)\psi\,dx \leq \int_\Omega |\nabla u|^{p-2}\nabla u \cdot \nabla \psi\, dx.
\end{equation*}To prove the claim, observe first that Caccioppoli's estimate allows us to conclude that
$$|\nabla u_\varepsilon|^{p-2}\nabla u_\varepsilon$$converges weakly in $L^{p/(p-1)}_{loc}(\Omega)$. Indeed, for any compact set $K \subset \Omega$, choose an open set $U \subset \Omega$ containing $K$ and a non-negative test function $0 \leq \xi \leq 1$ so that
\begin{equation*}
K \subset K':=\mbox{spt }\xi \subset U,
\end{equation*} and $\xi = 1$ in $K$. Then
\begin{equation}\label{in 3}
\int_K ||\nabla u_\varepsilon|^{p-2}\nabla u_\varepsilon|^{p/(p-1)}dx \leq \int_K|\nabla u_\varepsilon|^{p}dx \leq \int_\Omega |\nabla u_\varepsilon|^{p}\xi^{p}dx.
\end{equation}Observe that since $f$ satisfies \eqref{growth cond}, the lower term $f_\varepsilon$ verifies the  bound \eqref{est f epsilon}. Therefore, Lemma \ref{caccioppoli} applies and the right-hand side of \eqref{in 3} is bounded from above by
\begin{equation}\label{caccio}
C\left[(\textnormal{osc}_{K'}u_\varepsilon)^{p}\int_\Omega \left(|\nabla \xi|^{p} + 1\right) dx +\textnormal{osc}_{K'}u_\varepsilon\right],
\end{equation}Moreover, since $u_\varepsilon$ is an increasing sequence and converges pointwise to $u$ in $\Omega$, we have
$$\textnormal{osc}_{K'}u_\varepsilon \leq \sup_{K'}u-\inf_{K'}u_{\varepsilon_0},$$for all $\varepsilon < \varepsilon_0$. Then in view of \eqref{in 3}, \eqref{caccio} and the above comments, we can find a uniform bound for the integrals
$$\int_K ||\nabla u_\varepsilon|^{p-2}\nabla u_\varepsilon|^{p/(p-1)}dx,\quad \int_\Omega |\nabla u_\varepsilon|^{p}\xi^{p}dx.$$Hence $|\nabla u_\varepsilon|^{p-2}\nabla u_\varepsilon$ converges weakly in $L^{p/(p-1)}_{loc}(\Omega)$, and $\nabla u_\varepsilon$ converges weakly in $L^{p}_{loc}(\Omega)$. Since $u_\varepsilon$ converges pointwise to $u$, we derive that $u \in W^{1, p}_{loc}(\Omega)$ and $u_\varepsilon$ converges weakly in $W^{1,p}_{loc}(\Omega)$ to $u$. 

More can be said: $\nabla u_\varepsilon$ converges strongly in $L_{loc}^{p}(\Omega)$ to $\nabla u$. Indeed, take
$$\phi(x):=\left(u(x)-u_\varepsilon(x)\right)\theta(x), \quad x \in \Omega,$$where $\theta$ is a non-negative smooth test function compactly supported in $\Omega$. From
\begin{equation*}
\int_\Omega |\nabla u_\varepsilon|^{p-2}\nabla u_\varepsilon \cdot \nabla \phi\,dx \geq \int_\Omega f_\varepsilon(x, u_\varepsilon, \nabla u_\varepsilon)\phi\,dx,
\end{equation*}we get
\begin{equation}\label{weak conv}
\begin{split}
&\int_\Omega  \big[  |\nabla u|^{p-2}\nabla u -|\nabla u_\varepsilon|^{p-2}\nabla u_\varepsilon \big]\cdot \nabla (u-u_\varepsilon)\theta\,dx \\& \qquad \qquad\leq -\int_\Omega f_\varepsilon(x, u_\varepsilon, \nabla u_\varepsilon)\phi\,dx + \int_\Omega |\nabla u|^{p-2}\nabla u \cdot \nabla (u-u_\varepsilon)\theta\,dx.
\end{split}
\end{equation}By the weak convergence of  $u_\varepsilon $ to $ u$ in $W_{loc}^{1,p}(\Omega)$, the last integral in \eqref{weak conv} tends to $0$ as $\varepsilon \to 0$. The left-hand side is given by
\begin{equation}\label{weak con 2}
\begin{split}
& \int_\Omega  \theta\big[  |\nabla u|^{p-2}\nabla u -|\nabla u_\varepsilon|^{p-2}\nabla u_\varepsilon  \big]\cdot \nabla (u-u_\varepsilon)dx \\ & \qquad \qquad + \int_\Omega   (u-u_\varepsilon)\big[  |\nabla u|^{p-2}\nabla u -|\nabla u_\varepsilon|^{p-2}\nabla u_\varepsilon  \big]\cdot \nabla\theta\,dx.
\end{split}
\end{equation}The second integral in \eqref{weak con 2} is estimated in absolute value by

\begin{equation}
\|\nabla \theta\|_{L^\infty(\Omega)}\Big(\int_{\mbox{spt}\,\theta}|u-u_\varepsilon|^{p}dx\Big)^{1/p}\Big[\Big( \int_{\hbox{spt}\,\theta}|\nabla u|^{p}dx \Big)^{(p-1)/p} + \Big( \int_{\mbox{spt}\,\theta}|\nabla u_\varepsilon|^{p} dx\Big)^{(p-1)/p}\Big],
\end{equation}which tends to $0$ as $\varepsilon \to 0$. Moreover, since
\begin{equation}
\begin{split}
&-\int_{\hbox{spt}\,\theta}f_\varepsilon(x, u_\varepsilon, \nabla u_\varepsilon)\phi\, dx \leq \gamma_\infty\int_{\Omega}|\nabla u_\varepsilon|^{p-1}\left(u-u_\varepsilon\right)\theta dx + \|\phi\|_{L^\infty(\hbox{spt}(\theta))}\int_\Omega (u-u_\varepsilon)\theta\,dx,
\end{split}
\end{equation}with $\gamma_\infty:=\sup_{x \in\hbox{spt}(\theta)}|\gamma(u_\varepsilon(x))|$ does not depend on $\varepsilon$, it also holds that
$$\limsup_{\varepsilon \to 0}\Big[ -\int_{\mbox{spt}\,\theta}f_\varepsilon(x, u_\varepsilon, \nabla u_\varepsilon)\phi\,dx\Big] =0.$$Hence
\begin{equation}\label{strong}
\lim_{\varepsilon \to 0}\int_\Omega  \theta\big[  |\nabla u|^{p-2}\nabla u -|\nabla u_\varepsilon|^{p-2}\nabla u_\varepsilon  \big]\cdot \nabla (u-u_\varepsilon)dx  =0,
\end{equation}where we have used the fact that the integrand is always non-negative. Finally, using the inequality
$$2^{p-2}|\nabla u(x) - \nabla u_\varepsilon(x)|^{p} \leq  \big[  |\nabla u(x)|^{p-2}\nabla u(x) -|\nabla u_\varepsilon(x)|^{p-2}\nabla u_\varepsilon(x)  \big]\cdot \nabla (u(x)-u_\varepsilon(x))$$valid for all $p \geq 2$, and \eqref{strong}, we conclude the strong convergence of $\nabla u_\varepsilon$ in $L^p_{loc}(\Omega)$. 
Finally, \eqref{strong} together with \cite[Lemma 3.73]{H}), imply
$$|\nabla u_\varepsilon |^{p-2}\nabla u_\varepsilon \rightharpoonup |\nabla u|^{p-2}\nabla u \mbox{ in }L^{p/(p-1)}_{loc}(\Omega),$$and, in turn,  the strong convergence of the gradients $\nabla u_\varepsilon$ and Lemma \ref{convf} give
$$\lim_{\varepsilon \to 0}\int_\Omega f_\varepsilon(x, u_\varepsilon, \nabla u_\varepsilon)\psi\,dx = \int_\Omega f(x, u, \nabla u)\psi\,dx.$$
This ends the proof of the claim and we deduce that $u$ is a weak supersolution. 
\end{proof}

\subsection{The singular case: $1<p<2$}
Consider now the infimal convolution given in \eqref{infconv} choosing $q>\frac{p}{p-1}$, i.e.,
\begin{equation}\label{uInfSing}
u_\varepsilon(x):=\inf_{y\in\Omega}\left(u(y)+\frac{|x-y|^q}{q\varepsilon^{q-1}}\right).
\end{equation}
 Notice that $q>2$ for $1<p<2$.

We need the following auxiliar result, which is an adaptation of \cite[Lemma 4.3]{JJ}.
\begin{lemma}\label{negF}
Suppose that $u$ is a bounded viscosity supersolution to \eqref{1}. If there is $\hat{x}\in\Omega_{r(\varepsilon)}$ such that $u_\varepsilon$ is differentiable at $\hat{x}$ and $\nabla u_\varepsilon(\hat{x})=0$, then $f_\varepsilon(\hat{x},u_\varepsilon(\hat{x}),\nabla u_\varepsilon(\hat{x}))\leq 0$.
\end{lemma}

\begin{proof}
From \cite[Lemma 4.3]{JJ} we know that $u_\varepsilon(\hat{x})=u(\hat{x})$ and hence
$$u(y)+\frac{|\hat{x}-y|^q}{q\varepsilon^{q-1}}\geq u(\hat{x})\qquad \hbox{ for every }y\in\Omega.$$
Define
$$\psi(y):=u(\hat{x})-\frac{|\hat{x}-y|^q}{q\varepsilon^{q-1}},\qquad y\in\Omega,$$
which satisfies $\psi\in \mathcal{C}^2(\Omega)$, $\nabla \psi(\hat{x})=0$ and
\begin{equation}\label{limsupball}
\lim_{r\rightarrow 0}\sup_{y\in B_r(\hat{x})\setminus\{\hat{x}\}}(-\Delta_p\psi(y))=0,
\end{equation}
in view of $q > p/(p-1)$. Since $\psi(\hat{x})=u(\hat{x})$, $\psi(y)\leq u(y)$ for $y\in \Omega$, $\nabla \psi(x)\neq 0$ for all $x \neq \hat{x}$, and $u$ is a viscosity supersolution to \eqref{1},
$$\lim_{r\rightarrow 0}\sup_{y\in B_r(\hat{x})\setminus\{\hat{x}\}}(-\Delta_p\psi(y))\geq f(\hat{x},\psi(\hat{x}),\nabla\psi(\hat{x})).$$
Noticing that $\psi(\hat{x})=u_\varepsilon(\hat{x})$ and $\nabla\psi(\hat{x})=\nabla u_\varepsilon(\hat{x})=0$, by \eqref{limsupball} we conclude
$$0\geq f(\hat{x},\psi(\hat{x}),\nabla\psi(\hat{x}))=f(\hat{x},u_\varepsilon(\hat{x}),\nabla u_\varepsilon(\hat{x}))\geq f_\varepsilon (\hat{x},u_\varepsilon(\hat{x}),\nabla u_\varepsilon(\hat{x})).$$
\end{proof}

%
%

We can prove now Theorem \ref{nonSing} in the case $1<p<2$.

\begin{proof}[Proof of Theorem \ref{nonSing}]
Let $u_\varepsilon$ be defined in \eqref{uInfSing}. Proceeding as in the degenerate case, by Aleksandrov's theorem and Lemma \ref{infConvLemma}, 
$$-\Delta_p u_\varepsilon \geq f_\varepsilon (x,u_\varepsilon,\nabla u_\varepsilon)$$
a.e. in $\Omega_{r(\varepsilon)}\setminus \{\nabla u_\varepsilon =0\}$.
Performing the same approximation argument as in the proof of \cite[Theorem 4.1]{JJ} we reach that
$$\int_\Omega |\nabla u_\varepsilon|^{p-2}\nabla u_\varepsilon \cdot \nabla \psi\,dx\geq \int_{\Omega\setminus \{\nabla u_\varepsilon=0\}} f_\varepsilon (x,u_\varepsilon,\nabla u_\varepsilon)\psi\,dx,$$
for every $\psi\in \mathcal{C}_0^\infty(\Omega)$, $\psi\geq 0$, and therefore, since by Lemma \ref{negF} we know $f_\varepsilon\leq 0$ in the set $\{x\in\Omega:\;\nabla u_\varepsilon(x)=0\}$, we get
\begin{equation}\label{eqEps}
\int_\Omega |\nabla u_\varepsilon|^{p-2}\nabla u_\varepsilon \cdot\nabla \psi\,dx\geq \int_{\Omega} f_\varepsilon (x,u_\varepsilon,\nabla u_\varepsilon)\psi\,dx.
\end{equation}
Repeating the proof for the case $p\geq 2$ (and noticing that Lemma \ref{caccioppoli} works for every $1<p<\infty$) we obtain the uniform boundedness of $\nabla u_\varepsilon$ in $L^p_{loc}(\Omega)$ and
\begin{equation}\label{conv1}
\lim_{\varepsilon\rightarrow 0}\int_K[|\nabla u|^{p-2}\nabla u-|\nabla u_\varepsilon|^{p-2}\nabla u_\varepsilon]\cdot\nabla(u- u_\varepsilon)\,dx=0,
\end{equation}
for any compact set $K\subset\Omega$, and from here the convergence
\begin{equation}\label{weakConv}
|\nabla u_\varepsilon |^{p-2}\nabla u_\varepsilon \rightharpoonup |\nabla u|^{p-2}\nabla u \mbox{ in }L^{p/(p-1)}_{loc}(\Omega).
\end{equation}
Using H\"older's inequality and the vector inequality (see \cite[Chapter I]{DB})
$$\frac{|a-b|^2}{(|a|+|b|)^{2-p}}\leq C( |a|^{p-2}a-|b|^{p-2}b)\cdot (a-b),\qquad 1<p<2,$$
with $C=C(n,p)$ and $a,b\in\mathbb{R}^n$, we obtain
\begin{equation*}\begin{split}
\int_K|\nabla u-\nabla u_\varepsilon|^p\,dx&\leq \left(\int_K\frac{|\nabla u-\nabla u_\varepsilon|^2}{(|\nabla u|+|\nabla u_\varepsilon|)^{2-p}}\,dx\right)^{p/2}\left(\int_K(|\nabla u|+|\nabla u_\varepsilon|)^p\,dx\right)^{\frac{2-p}{2}}\\
&\leq C \left(\int_K\frac{|\nabla u-\nabla u_\varepsilon|^2}{(|\nabla u|+|\nabla u_\varepsilon|)^{2-p}}\,dx\right)^{p/2}\\
&\leq C \left(\int_K[|\nabla u|^{p-2}\nabla u-|\nabla u_\varepsilon|^{p-2}\nabla u_\varepsilon]\cdot\nabla(u-u_\varepsilon)\,dx\right)^{p/2}.
\end{split}\end{equation*}
Thus, from \eqref{conv1} we deduce that $\nabla u_\varepsilon$ converges to $\nabla u$ in $L^p_{loc}(\Omega)$. By \eqref{weakConv} and Lemma \ref{convf} we can pass to the limit in \eqref{eqEps} to conclude
$$\int_{\Omega}|\nabla u|^{p-1}\nabla u \cdot\nabla\psi\,dx\geq \int_\Omega f(x,u,\nabla u)\psi\,dx.$$
\end{proof}

\section{Proofs of Theorem \ref{converse} and Theorem \ref{nice converse}}\label{Proof converse}

\begin{proof}[Proof of Theorem \ref{converse} (i)]
Let $u  \in W^{1, \infty}_{loc}(\Omega)$ be a weak supersolution to \eqref{1}. To reach a contradiction, assume that $u$ is not a viscosity supersolution.  By assumption, there exist $x_0 \in \Omega$ and $\varphi \in \mathcal{C}^{2}(\Omega)$ so that $\nabla \varphi (x) \neq 0$, for all $x \neq x_0$, 
\begin{equation}\label{minn}
u(x_0)=\varphi(x_0), \quad u(x)>\varphi(x) \mbox{ for all $x\neq x_0$},
\end{equation}and
\begin{equation}\label{minn 2}
\lim_{r\to 0}\sup_{x \in B_r(x_0)\setminus \{x_0\}}\left(-\Delta_p \varphi(x)\right) < f(x_0, u(x_0), \nabla \varphi(x_0)).
\end{equation}
Moreover, by the $W^{1, \infty}_{loc}$-regularity of $u$, we may assume that $u$ is continuous in $\Omega$. Thus, the map
$$x \to f(x, u(x), \nabla \varphi(x))$$is continuous in $\Omega$, and \eqref{minn} yields

$$\lim_{r\rightarrow 0}\sup_{x \in B_r(x_0)\setminus \{x_0\}}\left[-\Delta_p \varphi(x)-f(x,u(x),\nabla\varphi(x))\right]<0.$$Hence, there exists some $r_0>0$ so that
\begin{equation}\label{eqq}
-\Delta_p\varphi\leq f(x,u(x),\nabla \varphi(x)),\qquad x\in B_{r_0}(x_0)\setminus\{x_0\}.
\end{equation}Let
$$m:=\inf_{\partial B_{r_0}(x_0)}\left( u-\varphi\right).$$Then by \eqref{minn}, $m >0$. Consider
$$\tilde{\varphi}(x):=\varphi(x) +m, \quad x \in \Omega.$$By \eqref{eqq},  $\tilde{\varphi}$ is a weak subsolution to
\begin{equation}\label{new eq}
-\Delta_p v = \tilde{f}(x, \nabla v),
\end{equation}in $ B_{r_0}(x_0)$, where $\tilde{f}(x, \eta):=f(x, u(x), \eta)$. Observe that $\tilde{f}$ is continuous in $\Omega \times \mathbb{R}^{n}$ and locally Lipschitz in $\eta$. Moreover, in the weak sense, we have
\begin{equation*}
-\Delta_p u  \geq f(x, u, \nabla u) = \tilde{f}(x, \nabla u),
\end{equation*}which shows that $u$ is a weak supersolution to \eqref{new eq}. In addition, $u \geq \tilde{\varphi}$ on $\partial B_{r_0}(x_0)$. By the Comparison Theorem \ref{comparison for weaks singular}, we conclude that $u \geq \tilde{\varphi}$ in $B_{r_0}(x_0)$. This contradicts \eqref{minn}.

\textit{Proof of Theorem \ref{converse} (ii)}  For a given weak supersolution $u \in \mathcal{C}^{1}(\Omega)$, by following the lines above and appealing to  the Comparison Theorem \ref{comparison for weaks}, we can show that $u$ is a viscosity supersolution in the non-critical set:
$$\{x \in \Omega: \nabla u(x) \neq 0\}.$$By \cite[Corollary 4.4]{JL2}, which holds true  for sub- and supersolutions, $u$ is a viscosity supersolution in the whole set $\Omega$.

\textit{Proof of Theorem \ref{converse} (iii)} The proof follows similarly just by using the assumption $\nabla u \neq 0$ in $\Omega$ together with the Comparison Theorem \ref{comparison for weaks}.
\end{proof}

\begin{remark}Observe that the $W^{1, \infty}_{loc}$-regularity  of $u$ is only needed to apply the comparison principles. In the rest of the proof, the continuity of $u$ would be enough.
\end{remark}

\begin{proof}[Proof of Theorem \ref{nice converse}] The result follows reproducing the proof of Theorem \ref{converse} using the Comparison Theorem \ref{nice case} instead of Theorems \ref{comparison for weaks singular} and \ref{comparison for weaks}.
\end{proof}

\section*{Appendix}

\subsection{Comparison principles for weak solutions} In this section we provide the comparison principles for weak solutions of \eqref{1} that we use in the proof of Theorem \ref{converse}. As we pointed out in the Introduction, other comparison results may be employed (see \cite{Tom}).

The first one is contained in  \cite[Corollary 3.6.3]{PS}. 

\begin{theorem}\label{comparison for weaks}

Assume that $f=f(x, s, \eta)$ is continuous in $\Omega \times \mathbb{R}\times \mathbb{R}^{n}$, non-increasing in $s$, and locally Lipschitz continuous with respect to $\eta$ in $\Omega \times \mathbb{R}\times \mathbb{R}^{n}$. Let $u \in W^{1, \infty}_{loc}(\Omega)$ be a weak supersolution and let $v \in W^{1, \infty}_{loc}(\Omega)$ be a weak subsolution to \eqref{1} in $\Omega$, for $1<p<\infty$. Assume that $|\nabla u|+ |\nabla v| > 0$ in $\Omega$. If $u \geq  v$ on $\partial \Omega$, then $ u \geq v$ in $\Omega$.

\end{theorem}

In the singular case, the assumptions on the gradients may be removed. See \cite[Corollary 3.5.2]{PS}.

\begin{theorem}\label{comparison for weaks singular}
Let $1<p \leq 2$. Assume that $f=f(x, s, \eta)$ is continuous in $\Omega \times \mathbb{R}\times \mathbb{R}^{n}$, non-increasing in $s$, and that it is locally Lipschitz continuous in $\eta$ on compact subsets of its variables. Then if  $u \in W^{1, \infty}_{loc}(\Omega)$ is a weak supersolution and if $v \in W^{1, \infty}_{loc}(\Omega)$ is a weak subsolution to \eqref{1} in $\Omega$ so that $u\geq v$ on $\partial \Omega$, then $u \geq v$ in $\Omega$.
\end{theorem}

Finally, in the case where $f$ does not depend on $\eta$, we have the following result (see \cite[Corollary 3.4.2]{PS}).

\begin{theorem}\label{nice case} Suppose that $f=f(x, s)$ is continuous in $\Omega \times \mathbb{R}$ and non-increasing in $s$. Let $u \in W^{1,p}_{loc}(\Omega) \cap \mathcal{C}(\Omega)$ be a supersolution and $v \in W^{1,p}_{loc}(\Omega) \cap \mathcal{C}(\Omega)$ a subsolution so that $u \geq v$ on $\partial \Omega$. Then $u \geq v$ in $\Omega$.

\end{theorem}

\begin{remark} In \cite{PS}, Theorem \ref{comparison for weaks} is stated in a more general framework of equations in divergence form as
\begin{equation}
\mbox{div}\left( A(x, u, \nabla u)\right)=f(x, u, \nabla u),
\end{equation}where the operator $A=A(x,s, \eta)$ is assumed to be continuous in $\Omega \times \mathbb{R}\times \mathbb{R}^{n}$, continuously differentiable with respect  to $s$ and $\eta$ for all $s$ and all $\eta \neq 0$, and elliptic in the sense that $\nabla A(x,s,\eta)$ is positive definite in $\Omega \times \mathbb{R}\times \left( \mathbb{R}^{n}\setminus \{0\}\right)$. In the particular case of the p-Laplace operators
$$A_p(\eta)=|\eta|^{p-2}\eta, \quad p \geq2,$$
\begin{equation*}\label{boundary problem2}
A_p(\eta)=  \left\lbrace
  \begin{array}{l}
     |\eta|^{p-2}\eta,  \text{ if } \eta \neq0\\
    0,  \qquad \,\,\,\text{ for } \eta=0, \\
  \end{array}
  \right.\quad 1<p\leq2,
\end{equation*}all of the assumptions above are satisfied. The positive definiteness of $\nabla A_p$ is a consequence of
$$\sum_{i,j=1}^{n}\frac{\partial A^{i}}{\partial \eta_j}(\eta)\xi_i\xi_j \geq c|\eta|^{p-2}|\xi|^{2},$$for a positive constant $c$. Finally, observe that $A$ is uniformly elliptic for $0 < |\eta| \leq C$ if $p\leq2$.  This allows the improved comparison result in Theorem \ref{comparison for weaks singular}.

\end{remark}

\section*{Acknowledgements}

The authors would like to thank the anonymous referee for her$\setminus$his comments.


\begin{thebibliography}{00}

\bibitem{CIL} M. Crandall, H. Ishii and P-L. Lions, User's guide to viscosity solutions of second order partial differential equations, \textit{Bull. of Amer. Soc.} \textbf{27}  1 (1992), 1-67.
\bibitem{CT} M. Cuesta and P. Tak\'a\v{c}, \textit{A strong comparison principle for positive solutions of degenerate elliptic equations}, Differential and Integral Equations \textbf{13} 4-6 (2000), 721-746.
\bibitem{DS} L. Damascelli and B. Sciunzi, \textit{Regularity, monotonicity and symmetry of positive solutions of m-Laplace equations}, Journal of Differential Equations \textbf{206} (2004), 483-515.


\bibitem{DB} E. DiBenedetto, \textit{$\mathcal{C}^{1+\alpha}$-local regularity of weak solutions of degenerate elliptic equations}, Nonlinear Anal. \textbf{7} (1983), 827-850.

\bibitem{H} J. Heinonen, T. Kilpelainen,  and O. Martio, \emph{Non-linear potential theory of degenerate elliptic equations}. Dover Publications, Inc., Mineola, NY, 2006. 

\bibitem{JJ} V. Julin, P. Juutinen, \textit{A new proof for the equivalence of weak and viscosity solutions for the p-Laplace equation}. Communications in PDE \textbf{37} 5 (2012), 934-946.
\bibitem{JLr} P. Juutinen and P. Lindqvist, \textit{A theorem of Rad\'o's type for the solutions of a quasi-linear equation}, Mathematical Research Letters \textbf{11} (2004), 31-34.
\bibitem{JL2} P. Juutinen and P. Lindqvist, \textit{Removability of a level set for solutions of quasilinear equations}, Communications in PDE \textbf{30} 3 (2005), 305-321.
\bibitem{JLM} P. Juutinen, P. Lindqvist, and J. Manfredi, \textit{On the equivalence of viscosity solutions and weak solutions for a quasilinear equation}, SIAM J. Math. Anal. \emph{33} 3 (2001), 699-717.\bibitem{N}N. Katzourakis, \textit{Nonsmooth convex functionals and feeble viscosity solutions of singular Euler–Lagrange equations}, Calculus of Variations and PDE \textbf{54} 1 (2015), 275-298.
\bibitem{Tom} T. Leonori, A. Porretta and G. Riey, \textit{Comparison principles for p-Laplace equations with lower order terms}, Annali di Matematica (2016), 1-27.
\bibitem{Lib} G. Lieberman, \textit{Boundary regularity for solutions of degenerate elliptic equations}, Nonlinear Anal. \textbf{12} 11 (1988), 1203-1219.
\bibitem{L} P. Lindqvist, \textit{On the definition and properties of p-superharmonic functions}, J. Reine Angew. Math. \textbf{365} (1986), 67-70.
\bibitem{MMS} S. Merch\'an, L. Montoro and B. Sciunzi, \textit{On the Harnack inequality for quasilinear elliptic equations with a first order  term}. Preprint.

\bibitem{PS} P. Pucci and J. Serrin, \textit{The  maximum principle}, Progress in non-linear differential equations and their applications Vol. 73 Birkh\"{a}user, Boston, 2007.
\bibitem{P} P. Pucci and R. Servadei, \textit{Regularity of weak solutions of homogeneous or inhomogeneous elliptic equations}, Indiana University Mathematical Journal \textbf{57}  7 (2008), 3329-3363.
\bibitem{RZ} J. Rakotoson and W. Ziemer, \textit{Local behaviour of solutions of quasilinear elliptic equations with general structure}, Transaction of the American Mathematical Society \textbf{319} 2 (1990), 747-764.




\end{thebibliography}
\end{document}